\documentclass[sn-mathphys,Numbered]{sn-jnl}

\usepackage{graphicx}%
\usepackage{multirow}%
\usepackage{amsmath,amssymb,amsfonts}%
\usepackage{amsthm}%
\usepackage{mathrsfs}%
\usepackage[title]{appendix}%
\usepackage{xcolor}%
\usepackage{textcomp}%
\usepackage{manyfoot}%
\usepackage{booktabs}%
\usepackage{algorithm}%
\usepackage{algorithmicx}%
\usepackage{algpseudocode}%
\usepackage{listings}%

\newtheorem{theorem}{Theorem}
\newtheorem{proposition}[theorem]{Proposition}%

\begin{document}

\title[Article Title]{Differential Operator Representation of $\mathfrak{sl\left( \mathrm{2,R} \right)}$ as Modules over Univariate and Bivariate Hermite Polynomials }

\author{\fnm{Manouchehr} \sur{Amiri}}
\affil{manoamiri@gmail.com}
\abstract{This paper presents the connections between univariate and bivariate Hermite polynomials and associated differential equations  with specific representations of Lie algebra sl(2,R)  whose Cartan sub-algebras coincide the associated differential operators of these differential equations . Applying the Baker-Campbell-Hausdorff formula to generators of these algebras, results in new relation for one-variable and Bivariate Hermite polynomials. A general form of sl(2,R) representation by differential operators and arbitrary polynomial basis such as Laguerre and Legendre polynomials is introduced.}

\maketitle

\section{Introduction}\label{sec1}Hermite polynomials are among the most applicable orthogonal polynomials. These polynomials arise in 
diverse fields of probability, physics, numerical analysis, and signal processing. As an example, in quantum mechanics, eigenfunction solutions to quantum harmonic oscillators are described in terms of Hermite polynomials. Bivariate Hermite polynomials are useful in algebraic geometry and two-dimensional quantum harmonic oscillators \cite{area2015zero},\cite{lorentz1990bivariate},\cite{area2013linear} and \cite{area2015bivariate}. Concerning the field of applications, Hermite polynomials in one variable are divided into probabilistic and physicist versions. In the present paper, we focus on one and two-dimensional probabilistic Hermite polynomials which are denoted as $H_{n}^e\left( x \right)$ and $H_{n,m}(x,y,\Lambda)$ respectively. We prove that the symmetries of associated differential equations and associated differential operators are compatible with  $\mathfrak{sl\left( \mathrm{2,R} \right)} $ algebra. By introducing isomorphic Lie algebras whose Cartan sub-algebras are the Hermite differential operators, applying the Baker-Campbell-Hausdorff formula yields new relations for univarite and bivariate Hermite polynomials in sections (2),(5). In section (6) a general form of  $\mathfrak{sl\left( \mathrm{2,R} \right)} $ representation by certain  differential operators and arbitrary polynomial basis including Laguerre and Legendre polynomials is introduced.

\section{ Associated Lie algebra of Hermite polynomials of one variable}\label{sec2}Probabilistic Hermite polynomials, is presented as:

\begin{equation}
H_{n}^e\left( x \right)=e^{-\frac{D^{2}}{2}}x^{n}\label{eq1}
\end{equation}

$H_{n}^e\left( x \right) 
$ are probabilistic Hermite polynomials or equivalently the solutions to Hermite differential equations:
\begin{equation}
\label{eq1}
\mathbb{D_{\mathrm{H}}}H_{n}^e\left( x \right)=\left( xD-D^{2} \right)H_{n}^e\left( x \right)=nH_{n}^e\left( x \right)\end{equation}\label{eq1}
\textup{where:}
\begin{equation}
D=\frac{d}{dx}
\end{equation}\

and $\mathbb{D_{\mathrm{H}}}=\left( xD-D^{2} \right)$ is denoted as Hermite differential operator. This is an eigenvalue problem with positive integer eigenvalues $n$. The equation (1)  is the transformation of basis $\left( 1,x,x^{2},x^{3},... \right) $ under the action of operator  $e^{-\frac{D^{2}}{2}}$.This is compatible with Rodrigues’ formula and results in probabilistic Hermite polynomials $H_{n}^e(x) $.

\textup{These monomials expand a polynomial vector space}
$\mathbb{V\mathfrak{}}$.
$\text{ The operator }e^{-\frac{D^{2}}{2}}
$ changes the basis $x^{n}$ into the basis   $H_{n}^e(x)$.
 Let ${\mathfrak{gl\left( \mathbb{V} \right)}}$ denote the linear transformation that maps  $\mathbb{V}$ onto itself.
We present isomorphic Lie algebras $  \mathfrak{sl\left( \mathrm{2,R} \right)}$ to be defined by $ \mathfrak{sl\left( \mathrm{2,R} \right)}$ modules on vector space $\mathbb{V}$
which is a linear map $  \phi $
 defined by $ \phi:\mathfrak{sl\left( \mathrm{2,R} \right)}\to {\mathfrak{gl\left( \mathbb{V}\right)}}$
 that preserves the commutator relations of
$\mathfrak{sl\left( \mathrm{2,R} \right)}$ \cite{post1996gl},\cite{howe2012non}.

\begin{equation}
\label{eq1}
\phi\left[ a,b \right]=\left[ \phi\left( a \right),\phi\left( b \right) \right]
\end{equation}\label{eq1}

 This representation is $\mathfrak{sl\left( \mathrm{2,R} \right)}$
 modules on vector space $\mathbb{V}$. 
 First, we review the structure of the irreducible vector field representation of $\mathfrak{sl\left( \mathrm{2,R} \right)}$.
 The generators of this algebra in matrix representation are as follows:
\begin{equation}
\label{eq1}
H=\frac{1}{2}\begin{bmatrix}
1 &0  \\
 0& -1
\end{bmatrix} \quad X=\begin{bmatrix}
0 &0  \\
 1& 0
\end{bmatrix} \quad Y=\begin{bmatrix}
0 &1  \\
 0& 0
\end{bmatrix}
\end{equation}

The commutation relations for this representation of $\mathfrak{sl\left( \mathrm{2,R} \right)}$ are:
\begin{equation}
\left[ X,Y \right]=2H  \quad  \left[ H,X \right]=-X  \quad  \left[ H,Y \right]=Y 
\end{equation}

\textup{Let's define a representation of $\mathfrak{sl\left( \mathrm{2,R} \right)}$  as its module on $\mathbb{V}$ that preserves commutation relations of the following differential operators as its generators \cite{post1996gl}:}

\begin{equation}
\textbf{h = }xD-\frac{n}{2}   \quad
\textbf{e = }D=\frac{\partial }{\partial x}   \quad
\textbf{f = }x^{2}D-nx
\label{eq1}\end{equation}\label{eq1}

With the same commutation relations
\begin{equation}
\left [ \textbf{e},\textbf{f} \right]=2 \textbf{h}   \quad
\left [ \textbf{h},\textbf{e} \right]=- \textbf{e}   \quad
\left [ \textbf{h},\textbf{f} \right]=\textbf{f}    
\end{equation}

The Cartan sub-algebra $\textbf{h}$ 
produces a decomposition of representation space:

\begin{equation}
\mathbb{V=}\oplus{V}_{j} 
\end{equation}

${V}_{j}$ are the eigenspaces (eigenfunctions) of generator $\textbf{h}$ as Cartan sub-algebra of $\mathfrak{sl\left( \mathrm{2,R} \right)}$ and provides the solutions to the related differential equation:

 \begin{equation}
\textbf{h}{V}_{j}=f\left( j \right){V}_{j}
\end{equation}

As an example, monomials $x^{n}$ are eigenfunctions or eigenspaces of generator $\textbf{h} $ realized as eigenspace ${V}_{j}$.
The eigenvalues $f(j)$ in most cases equals an integer $j$ or $ j\left( j+1 \right)$ as we observe in Hermite, Laguerre, and Legendre differential equations.
We search for a Lie algebra $\mathfrak{L_{\mathrm{H}}}$ isomorphic to $\mathfrak{sl\left( \mathrm{2,R} \right)})$ algebra that its generators to be defined by Hermite differential operators. Here we apply the transformation operator $e^{-\frac{D^{2}}{2}}$ as described in (1) for Hermite polynomials to derive similarity transformations (conjugation) of  $\mathfrak{sl\left( \mathrm{2,R} \right)}$ bases as  follows:

\begin{equation}
X_{1}=e^{-\frac{D^{2}}{2}}\textbf{h }e^{\frac{D^{2}}{2}}\quad
X_{2}=e^{-\frac{D^{2}}{2}}\textbf{e }e^{\frac{D^{2}}{2}}\quad
X_{3}=e^{-\frac{D^{2}}{2}}\textbf{f }e^{\frac{D^{2}}{2}} 
\end{equation}

Concerning to a theorem in Lie algebra theory, these generators constitute an isomorphic Lie algebra to $\mathfrak{sl\left( \mathrm{2,R} \right)})$ with similar commutation relations. We call this algebra as “ Hermite operator Lie algebra”. Due to the equation (1) that implies the change of basis $x^{n}$ to $H_{n}^e$, the operator $xD$  with eigenfunctions $x^n$ corresponds an operator   $\mathbb{D_{\mathrm{H}}}$  with eigenfunctions $H_{n}^e$  and common eigenvalues $n$ (we call $\mathbb{D_{\mathrm{H}}}$ as the Hermite differential operator) through a similarity transformation described by
\begin{equation}
\mathbb{D_{\mathrm{H}}}=e^{-\frac{D^{2}}{2}}\left( xD \right)e^{\frac{D^{2}}{2}}
\label{eq1}\end{equation}\label{eq1}$
$Therefore we have: 
\begin{equation}
X_{1}=e^{-\frac{D^{2}}{2}}\textbf{h }e^{\frac{D^{2}}{2}}=e^{-\frac{D^{2}}{2}}\left( xD-\frac{n}{2}  \right)e^{\frac{D^{2}}{2}}=
\mathbb{D_{\mathrm{H}}}-\frac{n}{2}
\end{equation}

Generator $X_{2}$ simply be calculated as $X_{2}=D$.  
\\

\begin{proposition}
For $X_{3}$ we have: 
\begin{equation}
X_{3}=\left( x-D \right)\left(\mathbb{D_{\mathrm{H}}}-n  \right)
\label{eq1}\end{equation}\

\end{proposition}

\begin{proof}
by equation (12) we have the identity:
\begin{equation}
\mathbb{D_{\mathrm{H}}}=\left( e^{-\frac{D^{2}}{2}}xe^{\frac{D^{2}}{2}} \right)
\left( e^{-\frac{D^{2}}{2}}De^{\frac{D^{2}}{2}} \right)=e^{-\frac{D^{2}}{2}}xe^{\frac{D^{2}}{2}}D
\label{eq1}\end{equation}\label{eq1}$
$Thus:
\begin{equation}
\mathbb{D_{\mathrm{H}}}D^{-1}=e^{-\frac{D^{2}}{2}}xe^{\frac{D^{2}}{2}}
\label{eq1}\end{equation}\label{eq1}$
$By equation (2) we have:
\begin{equation}
\left( xD-D^{2} \right)D^{-1}=\left( x-D \right)=e^{-\frac{D^{2}}{2}}xe^{\frac{D^{2}}{2}}
\label{eq1}\end{equation}\label{eq1}$
$Now for $X_{3}$ from (10) we have:
\begin{equation}
X_{3}=e^{-\frac{D^{2}}{2}}\textbf{f }e^{\frac{D^{2}}{2}}=e^{-\frac{D^{2}}{2}}(x^{2}D-nx)e^{\frac{D^{2}}{2}}=e^{-\frac{D^{2}}{2}}x^{2}De^{\frac{D^{2}}{2}}-ne^{-\frac{D^{2}}{2}}xe^{\frac{D^{2}}{2}}
\label{eq1}\end{equation}\label{eq1}$
$By equations (12) and (17) we get:
\begin{equation}
X_{3}=[e^{-\frac{D^{2}}{2}}xe^{\frac{D^{2}}{2}}][e^{-\frac{D^{2}}{2}}xDe^{\frac{D^{2}}{2}}]-n(x-D)=(x-D)\mathbb{D_{\mathrm{H}}}-n(x-D)=(x-D)(\mathbb{D_{\mathrm{H}}}-n)
\label{eq1}\end{equation}\
\end{proof}

The commutation relations of these bases which coincide with the Lie algebra $\mathfrak{sl\left( \mathrm{2,R} \right)}$ 
are as follows
\begin{equation}
\left[ X_{1},X_{2} \right]=-X_{2}  \quad  \left[ X_{1},X_{3} \right]=X_{3}  \quad  \left[ X_{2},X_{3} \right]=2X_{1} 
\end{equation}

\begin{proposition}
{Hermite polynomials satisfy the equation:} 
 \begin{equation}
e^{-\frac{D^{2}}{2}}(e^{-1}+x)^{n}=e^{(\mathbb{D\mathrm{_{H}}}-n-\frac{D}{1-e})}H_{n}^e 
\label{eq1}\end{equation}\
\end{proposition}

\begin{proof}
Due to a theorem for the BCH formula \cite{hall2015matrix}, \cite{matone2016closed}, if $\left[ X,Y \right]=sY $    for $s\in \mathbb{R}$ 
we have:
\begin{equation}
e^{X}e^{Y}=e^{X+\frac{s}{1-e^{-s}}Y}
\end{equation}

Respect to equation (20) for 

$\left[ X_{1},X_{2} \right]$, the BCH formula for $X_1$ and $ X_2$ generators gives:
\begin{equation}
e^{(xD-D^{2})}e^{D}=e^{[(xD-D^{2})-\frac{D}{1-e}]}
\end{equation}

The term $\frac{n}{2}$ in $X_1$ was omitted because it has no role in commutation relation of $\left[ X_{1},X_{2} \right]$.
Multiplying both sides by $H_n^e (x)$ gives:
\begin{equation}
e^{(xD-D^{2})}e^{D}H_n^e (x)=e^{[(xD-D^{2})-\frac{D}{1-e}]}H_n^e (x)
\end{equation}

For $e^{D} H_n^e (x)$ we obtain:
\begin{equation}
e^{D} H_n^e (x)=(1+D+\frac{D^{2}}{2}+ ...)H_n^e (x)
=H_n^e (x)+nH^{e}_{n-1}(x)+\frac{n(n-1))}{2!}H^{e}_{n-2}(x)+...+1
\end{equation}

Thus:
\begin{equation}
e^{D} H_n^e (x)=\sum_{k=0}^{n}\binom{n}{k}H^{e}_{n-k}(x)
\end{equation}

Substituting in equation (24) and replacing Hermite differential operator$(xD-D^{2})$ with $\mathbb{D\mathrm{_{H}}}$ gives:

\begin{equation}
e^{\mathbb{D_{\mathrm{H}}}}\sum_{k=0}^{n}\binom{n}{k}H^{e}_{n-k}(x)=e^{\left[ \mathbb{D_{\mathrm{H}}}-\frac{\mathrm{D}}{1-e} \right]}H^{e}_{n}(x)
\end{equation}

Then we have:

\begin{equation}
\sum_{k=0}^{n}\binom{n}{k}e^{\mathbb{D}_{\mathrm{H}}}H^{e}_{n-k}(x)=\sum_{k=0}^{n}\binom{n}{k}e^{n-k}H^{e}_{n-k}(x)=e^{\left[ \mathbb{D_{\mathrm{H}}}-\frac{\mathrm{D}}{1-e} \right]}H^{e}_{n}(x)
\end{equation}

 By identity $H_{n}^e\left( x \right)=e^{-\frac{D^{2}}{2}}x^{n}\label{eq1}$ we get:
\begin{equation}
\sum_{k=0}^{n}\binom{n}{k}e^{n-k}e^{-\frac{D^{2}}{2}}x^{n-k}=e^{\left[ \mathbb{D_{\mathrm{H}}}-\frac{\mathrm{D}}{1-e} \right]}H^{e}_{n}(x)
\end{equation}

\begin{equation}
e^{-\frac{D_{2}}{2}}\sum_{k=0}^{n}\binom{n}{k}(e^{-1})^{k}x^{n-k}=e^{\left[ \mathbb{D_{\mathrm{H}}}-n-\frac{\mathrm{D}}{1-e} \right]}H^{e}_{n}(x)
\end{equation}
 
\begin{equation}
e^{-\frac{D_{2}}{2}}(e^{-1}+x)^{n}=e^{\left[ \mathbb{D_{\mathrm{H}}}-n-\frac{\mathrm{D}}{1-e} \right]}H^{e}_{n}(x)
\end{equation}

\end{proof} 

\textup{Thus a new version of Hermite polynimial equation is as follows:} 

\begin{equation}
e^{-\left[ \mathbb{D_{\mathrm{H}}}-n-\frac{\mathrm{D}}{1-e} \right]}e^{-\frac{D_{2}}{2}}(e^{-1}+x)^{n}=H^{e}_{n}(x)
\end{equation}

\section{Bivariate Hermite Polynomials }\label{sec2}

\textup{An ordinary definition for bivariate Hermite polynomials is as follows \cite{area2015zero},\cite{area2015bivariate},\cite{area2013linear},\cite{appell1926fonctions}:}

\begin{equation}
 H_{n,m}(s,t,\Lambda)=(-1)^{m+n}e^{\varphi(s,t)/2}\frac{\partial^{m+n}}{\partial x^{m}{\partial x^{n}}}e^{-\varphi(s,t)/2} 
\end{equation}

With definition $\varphi(s,t)=as^{2}+2bst+ct^{2}$ as a positive definite quadratic form and $\Lambda=\begin{bmatrix}
a &b  \\
b & c
\end{bmatrix}$
\textup{, the equivalent equation for $H_{n,m}(s,t,\Lambda)$ reads as \cite{area2015zero},\cite{area2015bivariate},\cite{area2013linear},\cite{appell1926fonctions}:} 

\begin{equation}
H_{n,m}(s,t,\Lambda)=\sum_{k=0}^{min(n,m)}(-1)^{k}k!\binom{m}{k}\binom{n}{k}a^{\frac{n-k}{2}}b^{k}c^{\frac{m-k}{2}}H^{e}_{n-k}(\frac{as+bt}{\sqrt{a}}H^{e}_{m-k}(\frac{bs+ct}{\sqrt{c}})
\end{equation}
With   $ac-b^{2} >0$ ,  $a,c > 0$ .
\\

By changing variables $x=\frac{as+bt}{\sqrt{a}}$ and $y=\frac{bs+ct}{\sqrt{c}}$

we have:
\begin{equation}
\hat{H}_{n,m}(x,y,\Lambda)=\sum_{k=0}^{min(n,m)}(-1)^{k}k!\binom{m}{k}\binom{n}{k}a^{\frac{n-k}{2}}b^{k}c^{\frac{m-k}{2}}H^{e}_{n-k}(x)H^{e}_{m-k}(y)
\end{equation}

\textup{These polynomials satisfy the partial differential equation \cite{area2015zero}, \cite{erdelyi1981higher}:}

\begin{equation}
 \left[ (x\frac{\partial }{\partial x}-\frac{\partial^{2} }{\partial x^{2}})+(y\frac{\partial }{\partial y}-\frac{\partial^{2} }{\partial y^{2}})-2\frac{b}{\sqrt{ac}}\frac{\partial^{2} }{\partial x\partial y} \right]\hat{H}_{n,m}(x,y,\Lambda)=(m+n)\hat{H}_{n,m}(x,y,\Lambda)
\end{equation}
Let denote $\mathfrak{D}$   as the differential operator in equation (36) 

\begin{equation}
 \mathfrak{D}=(x\frac{\partial }{\partial x}-\frac{\partial^{2} }{\partial x^{2}})+(y\frac{\partial }{\partial y}-\frac{\partial^{2} }{\partial y^{2}})-2\frac{b}{\sqrt{ac}}\frac{\partial^{2} }{\partial x\partial y}
\end{equation}
If we denote $ \partial _{x}=\frac{\partial }{\partial x}$ and    $\partial _{y}=\frac{\partial }{\partial y}$ ,with the identities:

\begin{equation}
e^{-\frac{{\partial^{2}_{x} }}{2}}x^{m-k}=H^{e}_{m-k}(x) , 
e^{-\frac{{\partial^{2}_{y} }}{2}}y^{n-k}=H^{e}_{n-k}(y)
\end{equation}

Equation (35) changes to :

\begin{equation}
\hat{H}_{n,m}(x,y,\Lambda)=e^{-\frac{{\partial^{2}_{x} }}{2}}e^{-\frac{{\partial^{2}_{y} }}{2}}\sum_{k=0}^{min(n,m)}(-1)^{k}k!\binom{m}{k}\binom{n}{k}a^{\frac{n-k}{2}}b^{k}c^{\frac{m-k}{2}}x^{m-k}y^{n-k}
\end{equation}

We denote the new polynomials as $u_{n,m}(x,y)$  defined by:

\begin{equation}
u_{n,m}(x,y)=\sum_{k=0}^{min(n,m)}(-1)^{k}k!\binom{m}{k}\binom{n}{k}a^{\frac{n-k}{2}}b^{k}c^{\frac{m-k}{2}}x^{m-k}y^{n-k}
\end{equation}

\textup{If these polynomials are assumed as the linearly independent basis, the transformation from these bases to $\hat{H}_{n,m}(x,y,\Lambda)$ is as follows:}

\begin{equation}
\hat{H}_{n,m}(x,y,\Lambda)=e^{-\frac{{\partial^{2}_{x} }}{2}}e^{-\frac{{\partial^{2}_{y} }}{2}}u_{n.m}(x,y)=e^{-\frac{{\partial^{2}_{x} }+\partial^{2} _{y}}{2}}u_{n,m}(x,y)
\end{equation}

Therfore, the corresponding differential operator with $ u_{n,m}(x,y)$  as its eigenfunctions could be derived by similarity transformation:

\begin{equation}
\mathfrak{D'}=e^{\mathfrak{}\frac{{\partial^{2}_{x} }+\partial^{2}_{y}}{2}}\mathfrak{D}e^{-\frac{{\partial^{2}_{x} }+\partial^{2}_{y}}{2}}
\end{equation}

$\mathfrak{D}$  is denoted as the differential operator given in the eigenvalue equation (36).Thus:

\begin{equation}
\mathfrak{D'}=e^{\frac{{\partial^{2}_{x} }+\partial^{2}_{y}}{2}}\left[ \left( x\partial_{x}- \partial_{x} ^{2} \right)+\left( y\partial_{y}- \partial_{y} ^{2} \right)-2\frac{b}{\sqrt{ac}}\partial_{x}\partial_{y} \right]  e^{-\frac{{\partial^{2}_{x} }+\partial^{2}_{y}}{2}}
\end{equation}
\\

\begin{proposition}
   The reduced form of $\mathfrak{D'}$ is:
    
\begin{equation}
\mathfrak{D'}=x{\partial _{x}}+y{\partial _{y}}-2\frac{b}{\sqrt{ac}}{\partial _{x}}{\partial _{y}}
\end{equation}

\end{proposition}

    \begin{proof}        
    Respect to equation (36) and the commutativity of $\partial^{2}_{x}$ and $\partial^{2}_{y}$ we get: 

\begin{equation}
\mathfrak{D'}=e^{\frac{{\partial^{2}_{x} }}{2}}\left( x{\partial}_{x}-{\partial}_{x}^{2} \right) e^{-\frac{{\partial^{2}_{x} }}{2}}+e^{\frac{{\partial^{2}_{y} }}{2}}\left( x{\partial}_{y}-{\partial}_{y}^{2} \right) e^{-\frac{{\partial^{2}_{y} }}{2}}-2\frac{b}{\sqrt{ac}}\left(e^{\frac{{\partial^{2}_{x} }}{2}}{\partial}_{x}  e^{-\frac{{\partial^{2}_{x} }}{2}} \right)\left(e^{\frac{{\partial^{2}_{y} }}{2}}{\partial}_{y}  e^{-\frac{{\partial^{2}_{y} }}{2}} \right)
\end{equation}

\textup{With respect to equations (2) and (12)ans (16)  we have :}

\begin{equation}
\mathbb{D_{\mathrm{H}}}(x)=e^{\frac{{-\partial^{2}_{x} }}{2}} x\partial _{x} e^{\frac{{\partial^{2}_{x} }}{2}}
\label{eq1}\end{equation}\

\begin{equation}
e^{\frac{{\partial^{2}_{x} }}{2}}\mathbb{D_{\mathrm{H}}}(x) e^{\frac{{-\partial^{2}_{x} }}{2}}=e^{\frac{{\partial^{2}_{x} }}{2}}(x\partial_{x}-\partial^{2}_{x}) e^{\frac{{-\partial^{2}_{x} }}{2}}=x\partial_{x}
\end{equation}

\textup{Repeating for} $\mathbb{D_{\mathrm{H}}}(y)$ \textup{yields :} 

\begin{equation}
e^{\frac{{\partial^{2}_{y} }}{2}}(y\partial_{y}-\partial^{2}_{y}) e^{\frac{{-\partial^{2}_{y} }}{2}}=y\partial_{y}
\end{equation}

\textup{We have the identities: } $e^{\frac{{-\partial^{2}_{x} }}{2}} \partial _{x} e^{\frac{{\partial^{2}_{x} }}{2}}=\partial _{x}$ ,
$e^{\frac{{-\partial^{2}_{y} }}{2}} \partial _{y} e^{\frac{{\partial^{2}_{y} }}{2}}=\partial _{y}$

\textup{Then equation (37) reduces to:}

\begin{equation}
\mathfrak{D'}=x{\partial _{x}}+y{\partial _{y}}-2\frac{b}{\sqrt{ac}}{\partial _{x}}{\partial _{y}}
\end{equation}

\end{proof}

\textup{Therefore the differential operator} $\mathfrak{D'}$ \textup{satisfies the differential equation:}

\begin{equation}
\mathfrak{D'}u_{n,m}(x,y)=(m+n)u_{n,m}(x,y)
\end{equation}

\textup{The eigenvalues of} $\mathfrak{D'}$ \textup{are the same as the differential equation (36), because} $\mathfrak{D} $ \textup { and} $\mathfrak{D'}$ \textup{are related by the similarity relation (42).} 

\section{sl(2,R) Modules on Bivariate Hermite Polynomials }

\textup{In this section we introduce an associated Lie algebra of bivariate Hermite differential operator. First, we search for the compatible }$\mathfrak{sl\left( \mathrm{2,R} \right)}$ \textup{algebra in terms of differential operators of two variables. with respect to equations (6) and (7) the Cartan sub-algebra of sl(2,R) can be taken as:}  
\begin{equation}
\textbf{h}=\frac{1}{2}(x{\partial _{x}}+y{\partial _{y}}+1)+\alpha{\partial _{x}}{\partial _{y}}
\end{equation}

\textup{The additional term $\alpha{\partial _{x}}{\partial _{y}}$ has been chosen to satisfy the required commutation relations. The other generators are proposed as}

\begin{equation}
\textbf{e}=\alpha {\partial _{x}}{\partial _{y}}
\end{equation}

\begin{equation}
\textbf{f}=\frac{1}{2\alpha}xy+\frac{1}{2}(x\partial_{x}+y\partial_{y})+\frac{\alpha}{4}\partial_{x}\partial_{y}
\end{equation}

\textup{These generators satisfy the commutation relations of }$\mathfrak{sl\left( \mathrm{2,R} \right)}$ \textup{ as described in equation (7). By substituting }$\alpha=\frac{-b}{\sqrt{ac}}$ \textup{the differential operator} $\textbf{h}$ \textup{satisfies the differential equation:}

\begin{equation}
\textbf{h} u_{n,m}(x,y)=\left[ (x\partial_{x}+y\partial_{y}+1)-\frac{2b}{\sqrt{c}}\partial_{x}\partial_{y} \right]u_{n,m}(x,y)
\end{equation}

\textup{Respect to equations (43),(44) and (45) we have:}

\begin{equation}
\textbf{h} u_{n,m}(x,y)=\frac{1}{2}(m+n+1)u_{n,m}(x,y)
\end{equation}

\textup{Thus} $u_{n.m}(x,y)$ \textup{are eigenfunctions or weight vectors of} $\textbf{h}$ \textup{as the Cartan sub-algebra.   
According to the equation (39), With respect to equations (12) and (13), the similarity transformation of generators } $\textbf{h},$\textbf{e} \textup{and} $\textbf{f}$ \textup{by operator }$e^{-\mathfrak{}\frac{{\partial^{2}_{x} }+\partial^{2}_{y}}{2}}$ \textup{yields:}

\begin{equation}
\textbf{h}'=\frac{1}{2}(\mathbb{D}_{\mathrm{H}}(x)+\mathbb{D}_{\mathrm{H}}(y)+1)-\frac{b}{\sqrt{ac}}\partial_{x}\partial_{y}
\end{equation}

\begin{equation}
\textbf{e}'=-\frac{b}{\sqrt{ac}}\partial_{x}\partial_{y}
\label{eq1}\end{equation}\
\begin{equation}
\textbf{f}'=-\frac{\sqrt{ac}}{2b}(x-{\partial _{x}})(y-{\partial _{y}})+\frac{1}{2}(\mathbb{D}_{\mathrm{H}}(x)+\mathbb{D}_{\mathrm{H}}(y)+1)-\frac{b}{\sqrt{ac}}\partial_{x}\partial_{y}
\end{equation}

\textup{$\mathbb{D}_{\mathrm{H}}(x)$ and $\mathbb{D}_{\mathrm{H}}(y)$ are defined in equation (2) and (46). The bivariate Hermite polynomials $H_{n,m}(x,y)$ are eigenfunctions of $\textbf{h}'$ with eigenvalues $\frac{1}{2}(m+n+1)$.} 

\textup{The lowering operator in this algebra is given by:}
\begin{equation}
A^{-}=e'=-\frac{b}{\sqrt{ac}}\partial _{x}\partial _{y}
\end{equation}

$\textbf{h}'$\textup{ represents the Cartan subalgebra of related Lie algebra. One of the commutator relations is} 

\begin{equation}
\left[ \textbf{h}',A^{-}\right]=-A^{-}=-\frac{b}{\sqrt{ac}}\partial _{x}\partial _{y}
\end{equation}

\section{	BCH Formula and sl(2,R) Generators of Bivariate Hermite Polynomial}	

\begin{theorem}
    If in equation (35), $b$ be replaced by $b'=2b(1-e )$ and $\mathfrak{D}$ is defined by equation (37), the resultant bivariate Hermite  polynomial $\hat{H}_{n,m}(x,y,\Lambda')$ satisfies the equation :
    \begin{equation}
    e^{\mathfrak{D}}\hat{H}_{n,m}(x,y,\Lambda')=a^{n/2}c^{m/2}e^{m+n}H_{n}^{e}(x)H_{m}^{e}(y)
 \end{equation}
    
    \end{theorem}
    
    \begin{proof}
    \textup{Due to a theorem for the BCH formula \cite{hall2015matrix} , \cite{matone2016closed}, if  $\left[ X,Y \right]=sY$ then we have:}

    \begin{equation}
    e^{X}e^{Y}=e^{X+\frac{s}{1-e^{-s}}Y}
   \end{equation}

    \textup{we choose $X$ and $Y$ in such a way that the BCH formula is simplified to equations that gives rise to new relations of bivariate Hermite polynomials.}
    \textup{Let assume $X$ and $Y$ in the form of  bases introduced in equations (56) and (57) :}

\begin{equation}
    X=\mathfrak{D}=\mathbb{D}_{\mathrm{H}}(x)+\mathbb{D}_{\mathrm{H}}(y)-2\frac{b}{\sqrt{ac}}\partial _{x}\partial _{y}  
    \end{equation} 

 \begin{equation}
    Y=-\frac{2(1-e)b}{\sqrt{ac}}\partial _{x}\partial _{y}
    \end{equation}

   Where we used equation (46) for variables $X$ and $Y$.
     With respect to the commutation relation
     
    \begin{equation}
    \left [ X,Y \right ]=-Y=\frac{2(1-e)b}{\sqrt{ac}} \partial _{x}\partial _{y}
    \end{equation}
    
    \textup{Then, we have}
    \begin{equation}
    exp(X)exp(Y)=exp(X-\frac{Y}{1-e})=exp(X+\frac{2b}{\sqrt{ac}} \partial _{x}\partial _{y})=exp(\mathbb{D}_{\mathrm{H}}(x)+\mathbb{D}_{\mathrm{H}}(y)
    \end{equation}

    \begin{equation}
   exp[\mathbb{D}_{\mathrm{H}}(x)+\mathbb{D}_{\mathrm{H}}(y)-2\frac{b}{\sqrt{ac}}\partial _{x}\partial _{y}]exp[-\frac{2(1-e)b}{\sqrt{ac}}\partial _{x}\partial _{y}]=exp[\mathbb{D}_{\mathrm{H}}(x)+\mathbb{D}_{\mathrm{H}}(y)]
    \end{equation}

   \textup{ Multiplying both sides of (67) by $H_{n}^{e}(x)H_{m}^{e}(y)$ yields:}

    \begin{equation}
    e^{\left [ \mathbb{D}_{\mathrm{H}}(x)+\mathbb{D}_{\mathrm{H}}(y)-2\frac{b}{\sqrt{ac}}\partial _{x}\partial _{y} \right ]}e^{\left [ -\frac{2(1-e)b}{\sqrt{ac}}\partial _{x}\partial _{y} \right ]}H_{n}^{e}(x)H_{m}^{e}(y)=e^{\left [ \mathbb{D}_{\mathrm{H}}(x)+\mathbb{D}_{\mathrm{H}}(y) \right ]}H_{n}^{e}(x)H_{m}^{e}(y)
    \end{equation}
    
    \begin{equation}
    e^{\mathfrak{D}}\sum_{k=0}^{min(m,n)}(-1)^{n}k!\binom{n}{k}\binom{m}{k}\left [ \frac{2b(1-e)}{\sqrt{ac}} \right ]^{k}H_{n-k}^{e}(x)H_{m-k}^{e}(y)=e^{m+n}H_{n}^{e}(x)H_{m}^{e}(y) 
    \end{equation}

    \textup{Comparing this with equation (35) shows the replacement $b\to b'=2b(e-1)$ yields a new bivariate Hermite polynomial $\hat{H}_{n,m}(x,y,\Lambda')$ :} 

    \begin{equation}
    \hat{H}_{n,m}(x,y,\Lambda)=\sum_{k=0}^{min(n,m)}(-1)^{k}k!\binom{m}{k}\binom{n}{k}a^{\frac{n-k}{2}}b^{k}c^{\frac{m-k}{2}}H^{e}_{n-k}(x)H^{e}_{m-k}(y)
    \end{equation}

    \begin{equation}
    \hat{H}_{n,m}(x,y,\Lambda')=a^{n/2}c^{m/2}\sum_{k=0}^{min(n,m)}(-1)^{k}k!\binom{m}{k}\binom{n}{k}(\frac{b'}{\sqrt{ac}})^{k}H^{e}_{n-k}(x)H^{e}_{m-k}(y)
    \end{equation}

    \textup{Substitution of equation (69) into (71) gives:}
    \begin{equation}
    e^{\mathfrak{D}}\hat{H}_{n,m}(x,y,\Lambda')=a^{n/2}c^{m/2}e^{m+n}H_{n}^{e}(x)H_{m}^{e}(y) 
    \end{equation}
    
    \end{proof}

\section{	A General form of sl(2,R) Representation By Differential Operators and Polynomial Basis}

\textup{ Rodrigues' formula for some orthogonal polybnomials such as Hermite, Laguerre and Legendre polynomilas are defined by by the action of specific differential operators on the n-th integer power of some function $B(x)$ , with $n\in \mathbb{N}\cup \left\{ 0 \right\}$}:

\begin{equation}
    P_{n}(x)=\frac{1}{2^{n}n!}\frac{d ^{n}}{dx^{n}}(x^{2}-1)^{n} 
\end{equation}

\begin{equation}
   L_{n}(x)=\frac{1}{n!}\left( \frac{d }{dx}-1 \right)^{n}x^{n}
\end{equation}

\textup{The $B(x)$ in these equations are $B(x)=(x^2-1)$ and $B(x)=x$ for Legendre and Laguerre polynomials respectively. Let $B(x)$ and its integer powers form a set $\left\{ 1,B(x),B^{2}(x),... \right\}=\left\{ B^{n}(x) \right\}_{n\ge 0}$ of independent basis in the infinite dimensional polynomial space. One can interpret these formulas as transformations from polynomial space basis $\left\{ 1,B(x),B^{2}(x),... \right\}$ to new basis i.e Legendre, and Laguerre polynomials. The definition of $B(x)$ in Rodrigues' formula limited to a polynomial with degree at most 2. However in this section we use  $B(x)$ with no limitation and as any kind of smooth functions.} 
\\

\begin{proposition}

Let $\left\{ 1,B(x),B^{2}(x),... \right\}=\left\{ B^{n}(x) \right\}_{n\ge 0}$ constitutes a basis for polynomial space. A differential operator representation of  $\mathfrak{sl\left( \mathrm{2,R} \right)}$ modules on these polynomials  is represented as 
\end{proposition}

 \begin{equation}
    \textbf{h}=\frac{B}{B'}D-\frac{n}{2}  \quad \textbf{e}=\frac{D}{B'} \quad \textbf{f}=\frac{B^{2}}{B'}D-nB
    \end{equation}

    \textup{Where $B'(x)$ is the derivative of $B(x)$ and $D$ is derivative operator respect to $x$.}

    \begin{proof}
    \textup{It is straightforward to prove that these bases satisfy the commutation relations of $\mathfrak{sl\left( \mathrm{2,R} \right)}$ in equation (8).}    
    \end{proof}

      \textup{The polynomials $B^{n}(x)$ are eigenfunctions of the operator $\textbf{h}$ as Cartan sub-algebra with integer eigenvalues and as the author proved \cite{Amiri:2023sfs}, similarity transformation of $\textbf{h}$ with operators defined in equations (74) and (75) yields the Legendre and Laguerre differential operators and equations respectively. Solutions of these differential equations are the corresponding eigenfunctions. We showed that the BCH formula mentioned in equation (62) holds when the commutation relation is $\left[ X,Y \right]=sY$ . Therefor this formula also holds for bases $\textbf{h}$ and $\textbf{e}$ in equation (75) :}

\begin{equation}
    \left[ \textbf{h},\textbf{e} \right]=-\textbf{e}
\end{equation}

    \textup{Let $\textbf{e}'=(1-e)\textbf{e} $ ,where $e$ is Euler number and $\textbf{h}'=\textbf{h}+\textbf{e}$. Then the commutation relation becomes $\left[\textbf{h}' ,\textbf{e}' \right]=-(1-e)\textbf{e}=-\textbf{e}'$ and the BCH formula reads as:}

\begin{equation}
    exp(\textbf{h}')exp(\textbf{e}')=exp\left(\textbf{h}'-\frac{\textbf{e}'}{1-e}  \right)=exp(\textbf{h})
  \end{equation}

\begin{equation}
    exp\left( \frac{B}{B'}D+\frac{D}{B'} \right)exp\left( (1-e)\frac{D}{B'} \right)=exp\left( \frac{B}{B'}D \right)
  \end{equation}

    \textup{Substituting  $B$ in equation (78) by any polynomials of the form of Laguerre and Legendre polynomial  that is compatible with operator relations in equations (73) and (74), results in new relations. }

\section{Conclusion}

By introducing isomorphic Lie algebras to $\mathfrak{sl\left( \mathrm{2,R} \right)}$ whose generators are defined by differential operators involved in univariate and bivariate Hermite differential equations, we showed that these Hermite polynomials, are eigefunctions of Cartan suba-lgebras of $\mathfrak{sl\left( \mathrm{2,R} \right)}$ . By using the commutation relations of these algebras and the BCH formula, new relations for Hermite polynomials are obtained. A general form of sl(2,R) representation  by differential operators and polynomial basis such as Laguerre, and Legendre differential polynomials, is introduced. 

\bibliography{Reference}


\section*{Declarations}

\begin{itemize}

\item \textup{The author confirms sole responsibility for the study conception and manuscript preparation.}
\item \textup{No funding was received for conducting this study.}
\item \textup{The author declares that he has no conflict of interest.}
\end{itemize}



\end{document}